\documentclass[a4paper,leqno]{amsart}

\usepackage{amsmath}
\usepackage{amsthm}
\usepackage{amssymb}
\usepackage{amsfonts}
\usepackage[applemac]{inputenc}
\usepackage{mathrsfs}

\def\C{{\mathbb C}}
\def\R{{\mathbb R}}
\def\N{{\mathbb N}}

\def\Sph{{\mathbb S}} 

\def\virgp{\raise 2pt\hbox{,}}

\def\({\left(}
\def\){\right)}
\def\<{\left\langle}
\def\>{\right\rangle}

\def\d{{\partial}}

\def\eps{\varepsilon}

\DeclareMathOperator{\supp}{supp}

\newcommand{\de}{\mathrm{d}}
\newcommand{\re}{\mathrm{Re}}
\newcommand{\im}{\mathrm{Im}}
\newcommand{\cfun}{\mathcal C}
\newcommand{\mez}{\frac{1}{2}}
\newcommand{\1}[1]{\frac{1}{#1}}
\renewcommand{\d}{\partial}

\newcommand{\conj}[1]{\overline{#1}}

\newcommand{\pare}[1]{\left(#1\right)}
\newcommand{\quadre}[1]{\left[#1\right]}
\newcommand{\curly}[1]{\left\{#1\right\}}

\newcommand{\lplq}[4]{\|#1\|_{L^{#2}_tL^{#3}_x(#4\times\R^2)}}

\newcommand{\lplqs}[3]{\|#1\|_{L^{#2}_tL^{#3}_x}}
\newcommand{\LpLq}[4]{\norma{#1}_{L^{#2}_tL^{#3}_x\pare{#4\times\R^2}}}
\newcommand{\slabint}{\int_{k\tau}^{(k+1)\tau}\int_{\R^2}}
\newcommand{\wlim}{\rightharpoonup}
\newcommand{\wstar}{\overset{\ast}{\wlim}}
\newcommand{\norma}[1]{ \left\|#1\right\|}
\newcommand{\scalar}[2]{\left\langle#1,#2\right\rangle}

\DeclareMathOperator{\diver}{div}

\theoremstyle{plain}
\newtheorem{theorem}{Theorem}[section]
\newtheorem{lemma}[theorem]{Lemma}
\newtheorem{corollary}[theorem]{Corollary}
\newtheorem{proposition}[theorem]{Proposition}

\theoremstyle{definition}
\newtheorem{definition}[theorem]{Definition}

\newtheorem{remark}[theorem]{Remark}
\newtheorem*{remark*}{Remark}

\numberwithin{equation}{section}

\begin{document}

\title[QHD system in 2D]
{The Quantum Hydrodynamics system in two space dimensions}
\author[P. Antonelli]{Paolo Antonelli}
\address[P. Antonelli]{Department of Applied Mathematics and
Theoretical Physics\\
CMS, Wilberforce Road\\ Cambridge CB3 0WA\\ England}
\email{p.antonelli@damtp.cam.ac.uk}
\author[P. Marcati]{Pierangelo Marcati}
\address[P. Marcati]{Department of Pure and Applied Mathematics\\
via Vetoio, 1\\ 67010 Coppito\\ L'Aquila, Italy}
\email{marcati@univaq.it}
\begin{abstract}
In this paper we study global existence of weak solutions for the Quantum Hydrodynamics System in 2-D in the space of energy. We do not require any additional regularity and/or smallness assumptions on the initial data. Our approach replaces the WKB formalism with a polar decomposition theory which is not limited by the presence of vacuum regions. In this way we set up a self consistent theory, based only on particle density and current density, which does not need to define velocity fields in the nodal regions. The mathematical techniques we use in this paper are based on uniform (with respect to the approximating parameter) Strichartz estimates and the local smoothing property.
\end{abstract}

\date{\today}

\subjclass[2000]{....}
\keywords{....}

\thanks{The first author is supported by Award No. KUK-I1-007-43, funded by the King Abdullah University of Science and Technology (KAUST).}
\maketitle
\section{Introduction}
In this paper we are concerned with the Cauchy problem for the Quantum Hydrodynamics System in two space dimensions
\begin{equation}\label{eq:QHD}
\left\{\begin{array}{l}
\d_t\rho+\diver J=0\\
\d_tJ+\diver\pare{\frac{J\otimes J}{\rho}}+\nabla P(\rho)+\rho\nabla V
+f(\sqrt{\rho}, J, \nabla\sqrt{\rho})
=\frac{\hbar^2}{2}\rho\nabla\pare{\frac{\Delta\sqrt{\rho}}{\sqrt{\rho}}}\\
-\Delta V=\rho-C(x),
\end{array}\right.
\end{equation}
where $t\geq0$, $x\in\R^3$, with initial data
\begin{equation}\label{eq:QHD_IV}
\rho(0)=\rho_0,\quad J(0)=J_0.
\end{equation}
The unknowns $\rho, J$ represent the charge and the current densities, respectively, $P(\rho)$ is the classical pressure, which we assume to satisfy the usual 
$\gamma-$law, $P(\rho)=\frac{p-1}{p+1}\rho^{\frac{p+1}{2}}$, with $p$ to be specified later. $V$ is the self-consistent electric potential, given by the Poisson equation. The term 
$\frac{\hbar^2}{2}\rho\nabla\pare{\frac{\Delta\sqrt{\rho}}{\sqrt{\rho}}}$ can be interpreted as the quantum Bohm potential, or as a quantum correction to the pressure. Indeed with some regularity assumption we can write the dispersive term in different ways:
\begin{equation}\label{eq:bohm}
\frac{\hbar^2}{2}\rho\nabla\pare{\frac{\Delta\sqrt{\rho}}{\sqrt{\rho}}}
=\frac{\hbar^2}{4}\diver(\rho\nabla^2\log\rho)
=\frac{\hbar^2}{4}\Delta\nabla\rho-\hbar^2\diver(\nabla\sqrt{\rho}\otimes\nabla\sqrt{\rho}).
\end{equation}
There is a formal analogy  between \eqref{eq:QHD} and the classical fluid mechanics system, in particular  when $\hbar  = 0$, the system  \eqref{eq:QHD}  formally coincides with the (nonhomogeneous) Euler-Poisson incompressible fluid system. 
\newline
The theoretical description of microphysical systems is generally based on the wave mechanics of 
Schr\"odinger , the matrix mechanics of Heisenberg or the path-integral mechanics of  Feynman (see \cite{F}).  Another approach to Quantum mechanics was taken by Madelung and  de Broglie (see \cite{M}), in particular the hydrodynamic theory of quantum mechanics  has been later extended by de Broglie (the idea of  “double solution”) and used as a scheme for quasicausal interpretation of microphysical systems. 
There is an extensive literature (see for example \cite{DGPS}, \cite{L}, \cite{K}, \cite{KB}, and references therein) where  superfluidity phenomena are described by means of quantum hydrodynamic  systems. 
Furthermore, the Quantum Hydrodynamics system is well known in literature since it has been used in modeling semiconductor devices at nanometric scales (see \cite{G}). 
 The hydrodynamical formulation for quantum mechanics is quite useful with respect to other descriptions for semiconductor devices, such as those based on Wigner-Poisson or  Schr\"odinger-Poisson, since kinetic or Schr\"odinger equations are computationally very expensive. For a derivation of the QHD system we refer to \cite{AI}, \cite{JMM}, \cite{GM}, \cite{DGM1}.
\newline
In \eqref{eq:QHD} the term $f(\sqrt{\rho}, J, \nabla\sqrt{\rho})$ represents collisions between electrons in models for semiconductor devices, see for example \cite{BW} where $f=J$. Otherwise, more generally it is used to mimic the interaction of the quantum fluid with an external fluid. In this paper we will treat the case $f=J$, whereas we will consider the more general case in a forthcoming paper.
\newline
We are interested in studying the global existence of weak solutions to the system in the class of finite energy initial data without higher regularity hypotheses or smallness assumptions. Local existence of solutions to \eqref{eq:QHD} is proved in \cite{JMR}, while global existence with further regularity for the initial data is proved in \cite{ML}. For the three dimensional case the problem of proving existence of weak solutions in the space of energy is solved in \cite{AM}.
\newline
There is a formal equivalence between the system \eqref{eq:QHD}
and the following nonlinear Schr\"odinger-Poisson system
\begin{equation}\label{eq:NLS_arg}
 \left\{\begin{array}{l}
 i\hbar\d_t\psi+\frac{\hbar^2}{2}\Delta\psi=|\psi|^{p-1}\psi+V\psi+\tilde V\psi\\
 -\Delta V=|\psi|^2
 \end{array}\right.
 \end{equation}
 where $\tilde V=\1{2i}\log\pare{\frac{\psi}{\conj{\psi}}}$, in particular the hydrodynamic system \eqref{eq:QHD} can be obtained by defining $\rho=|\psi|^2$, 
 $J=\hbar\im\pare{\conj{\psi}\nabla\psi}$ and by computing the related balance laws.
 \newline
 This problem has to face a serious mathematical difficulty connected with the need to solve 
 \eqref{eq:NLS_arg} with the ill-posed potential $\tilde V$. 
 \newline
 On the other hand, in two space dimensions we have to face another mathematical difficulty, which is the singularity of the electrostatic potential $V$. Indeed, the Green's function for the Poisson equation in $\R^2$ reads
 \begin{equation}
 \Phi(x):=-\frac{1}{2\pi}\log|x|,\qquad x\in\R^2,
 \end{equation}
 and consequently the electrostatic potential is implicitly given by the following convolution
 \begin{equation}\label{eq:elec_pot}
V(t, x):=-\frac{1}{2\pi}\int_{\R^2}\log|x-y|\rho(t, y)\de y.
\end{equation}
Hence $V$ could be unbounded at infinity, even if we assume $\rho$ to be sufficiently decaying at infinity. Consequently some additional hypotheses on the electrostatic potential at the initial time are required in order to assure its boundedness and some suitable integrability properties.
\newline
A natural framework to study the existence of the weak solutions to \eqref{eq:QHD} is given by the space of finite energy states. Here the energy associated to \eqref{eq:QHD} is defined by
\begin{equation}\label{eq:en}
E(t):=\int_{\R^2}\frac{\hbar^2}{2}|\nabla\sqrt{\rho(t)}|^2+\mez|\Lambda(t)|^2+f(\rho(t))
+\mez|\nabla V(t)|^2dx,
\end{equation}
where $\Lambda:=J/\sqrt{\rho}$, and $f(\rho)=\frac{2}{p+1}\rho^{\frac{p+1}{2}}$. The function $f(\rho)$ denotes the internal energy, which is related to the pressure through the identity 
$P(\rho)=\rho f'(\rho)-f(\rho)$.
\newline
Therefore our initial data are required to satisfy $E(0)<\infty$ 
and to have finite mass, i.e. $\int\rho_0dx<\infty$, or equivalently (by the Gagliardo-Nirenberg inequality), they are required to be such that
\begin{equation}
\sqrt{\rho_0}\in H^1(\R^2)\quad\textrm{and}\quad \Lambda_0:=J_0/\sqrt{\rho_0}\in L^2(\R^2).
\end{equation}
This paper is organized as follows: in Section 2 we present the main theorem, and the notations and main mathematical tools used throughout the paper. In Section 3 we recall the polar decomposition of a wave function into its amplitude and a unitary factor, as presented in \cite{AM}. In Section 4 we give a result of global existence of solutions to a nonlinear Schr\"odinger-Poisson system in two dimensions, and explain its equivalence with the QHD system \eqref{eq:QHD} without a collision term. In Section 5 we present how to construct a sequence of approximate solutions to \eqref{eq:QHD}, and show the consistency for this approximate solutions. In Section 6 we prove some a priori bounds for the sequence of approximate solutions in order to pass to the limit, prove the convergence and showing the existence of weak solutions to \eqref{eq:QHD}.
\section{Statement of the main theorem and preliminaries}
First of all, let us recall the definition of a weak solution to \eqref{eq:QHD}, \eqref{eq:QHD_IV}.
 \begin{definition}
 We say the pair $(\rho,J)$ is a \emph{weak solution} to the Cauchy problem \eqref{eq:QHD}, 
 \eqref{eq:QHD_IV} in $[0, T)\times\R^2$ with locally integrable initial data $(\rho_0, J_0)$,
  if there exist $\sqrt{\rho}, \Lambda$, such that 
$\sqrt{\rho}\in L^2_{loc}([0, T);H^1_{loc}(\R^2))$, 
 $\Lambda\in L^2_{loc}([0, T);L^2_{loc}(\R^2))$ and
  by defining $\rho:=(\sqrt{\rho})^2$, $J:=\sqrt{\rho}\Lambda$, one has
   \begin{itemize}
 \item for any test function 
  $\eta\in\cfun_0^\infty([0,T)\times\R^2)$ we have
 \begin{equation}\label{eq:QHD1}
 \int_0^T\int_{\R^2}\rho\d_t\eta+J\cdot\nabla\eta dxdt+\int_{\R^2}\rho_0\eta(0)dx=0;
 \end{equation}
 \item for any test function $\zeta\in\cfun_0^\infty([0,T)\times\R^2;\R^2)$
 \begin{multline}\label{eq:QHD2}
 \int_0^T\int_{\R^2}J\cdot\d_t\zeta+\Lambda\otimes\Lambda:\nabla\zeta+P(\rho)\diver\zeta
 -\rho\nabla V\cdot\zeta-J\cdot\zeta\\
 +\hbar^2\nabla\sqrt{\rho}\otimes\nabla\sqrt{\rho}:\nabla\zeta
 -\frac{\hbar^2}{4}\rho\Delta\diver\zeta dxdt+\int_{\R^2}J_0\cdot\zeta(0)dx=0;
 \end{multline}
 \item \emph{(generalized irrotationality condition)} the identity 
 \begin{equation}\label{eq:irrot}
 \nabla\wedge J=2\nabla\sqrt{\rho}\wedge\Lambda
 \end{equation}
 holds in the sense of distributions.
 \end{itemize}
 \end{definition}
 \begin{remark}
 Suppose we are in the smooth case, so that we can factorize $J=\rho u$, for some current velocity field $u$, then the last condition \eqref{eq:irrot} simply means $\rho\nabla\wedge u=0$, i.e. the current velocity $u$ is irrotational in 
 $\rho\de x$. This is why we will call it \emph{generalized irrotationality condition}.
 \end{remark}
 \begin{definition}
 We say that the weak solution $(\rho, J)$ to the Cauchy problem \eqref{eq:QHD}, \eqref{eq:QHD_IV} is a \emph{finite energy weak solution} (FEWS) in $[0, T)\times\R^2$, if in addition the energy \eqref{eq:en} is finite for almost every $t\in[0, T)$.
 \end{definition}
The hydrodynamic structure of the system \eqref{eq:QHD}, \eqref{eq:QHD_IV} should not lead to conclude that its solutions behave like classical fluids, indeed the connection with Schr\"odinger equations suggests that $(\rho_0, J_0)$ should in any case be seen as momenta related to some wave function $\psi_0$. The main result of this paper is the global existence of FEWS by assuming only the initial data $(\rho_0, J_0)$ are momenta of a wave function $\psi_0\in H^1(\R^2)$, without smallness or smoothness additional conditions. Furthermore, as already pointed out above, because of the singularity of the electrostatic potential, we have to make an additional assumption on the initial mass density $\rho_0$. We can now state the main theorem of this paper.
\begin{theorem}\label{thm:2D_QHD}
For any $\psi_0\in H^1(\R^2)$ let us denote
\begin{equation*}
\rho_0=|\psi_0|^2,\qquad J_0=\hbar\im(\conj{\psi_0}\nabla\psi_0),
\end{equation*}
and furthermore assume that
\begin{equation}
\int_{\R^2}\rho_0\log\rho_0dx<\infty,
\end{equation}
and
\begin{equation}
V(0,x):=-\1{2\pi}\int_{\R^2}\log|x-y|\rho_0(y)dy\quad\in L^r(\R^2),
\end{equation}
for some $2<r<\infty$. Then there exists a FEWS of \eqref{eq:QHD}, \eqref{eq:QHD_IV}, in $[0, T)\times\R^2$.
\end{theorem}
\subsection{Notations}
For convenience of the reader we will recall some notations which will be used in the sequel.
\newline
If $X$, $Y$ are two quantities (typically non-negative), we use $X\lesssim Y$ to denote $X\leq CY$, for some absolute constant $C>0$.
\newline
We will use the standard Lebesgue norms for complex-valued measurable functions $f:\R^d\to\C$ 
\begin{equation*}
\|f\|_{L^p(\R^d)}:=\pare{\int_{\R^d}|f(x)|^p\de x}^{\1{p}}.
\end{equation*}
If we replace $\C$ by a Banach space $X$, we will adopt the notation
\begin{equation*}
\|f\|_{L^p(\R^d;X)}:=\pare{\int_{\R ^d}\|f(x)\|_{X}\de x}^{1/p}
\end{equation*}
to denote the norm of $f:\R^d\to X$. In particular, if X is a Lebesgue space $L^r(\R^n)$, and $d=1$, we will shorten the notation by writing
\begin{equation*}
\|f\|_{L^q_tL^r_x(I\times\R^n)}:=\pare{\int_I\|f(t)\|_{L^r(\R^n)}^q\de t}^{1/q}
=\pare{\int_I(\int_{\R^n}|f(t,x)|^r\de x)^{q/r}\de t}^{1/q}
\end{equation*}
to denote the mixed Lebesgue norm of $f:I\to L^r(\R^n)$; moreover, we will write 
$L^q_tL^r_x(I\times\R^n):=L^q(I;L^r(\R^n))$.
\newline
For $s\in\R$ we will define the Sobolev space $H^s(\R^n):=(1-\Delta)^{-s/2}L^2(\R^n)$.
 \subsection{Nonlinear Schr\"odinger equations setting}
\begin{definition}[see \cite{Caz}]
We say that $(q, r)$ is an \emph{admissible pair} of exponents if $2\leq q\leq\infty, 2\leq r<\infty$, and
\begin{equation}
\frac{1}{q}+\frac{1}{r}=\frac{1}{2}.
\end{equation}
\end{definition}
Now we will introduce the \emph{Strichartz norms}. For more details, we refer the reader to \cite{T}. Let $I\times\R^2$ be a space-time slab, we define the Strichartz norm $\dot S^0(I\times\R^2)$
\begin{equation*}
\|u\|_{\dot S^0}(I\times\R^2):=\sup\|u\|_{L^q_tL^r_x(I\times\R^2)},
\end{equation*}
where the $\sup$ is taken over all the admissible pairs $(q, r)$. For any 
$k\geq1$, we can define
\begin{equation*}
\|u\|_{\dot S^k(I\times\R^2)}:=\|\nabla^ku\|_{\dot S^0(I\times\R^2)}.
\end{equation*}
 \begin{theorem}[Keel, Tao \cite{KT}]\label{thm:strich}
 Let $(q,r), (\tilde q, \tilde r)$ be two arbitrary admissible pairs of exponents, and let $U(\cdot)$ be the free Schr\"odinger group. Then we have
 \begin{align}
 \|U(t)f\|_{L^q_tL^r_x}&\lesssim\|f\|_{L^2(\R^2)}\\
 \|\int_{s<t}U(t-s)F(s)\de s\|_{L^q_tL^r_x}&\lesssim\|F\|_{L^{\tilde q'}_tL^{\tilde r'}_x}\\
 \|\int U(t-s)F(s)\de s\|_{L^q_tL^r_x}&\lesssim\|F\|_{L^{\tilde q'}_tL^{\tilde r'}_x}.
 \end{align}
 \end{theorem}
By using the previous Theorem, we can state the following Lemma about some further integrability properties for solutions of the Schr\"odinger equations.
 \begin{lemma}
Let $I$ be a compact interval, and let $u:I\times\R^2\to\C$ be a Schwartz solution to the 
Schr\"odinger equation
\begin{equation*}
i\d_tu+\Delta u=F_1+\dotsc+F_M,
\end{equation*}
for some Schwartz functions $F_1, \dotsc, F_M$. Then we have
\begin{equation*}
\|u\|_{\dot S^0(I\times\R^n)}\lesssim\|u(t_0)\|_{L^2(\R^n)}+\|F_1\|_{L^{q_1'}_tL^{r_1'}_x(I\times\R^2)}+\dotsc
+\|F_M\|_{L^{q_M'}_tL^{r_M'}_x(I\times\R^2)},
\end{equation*}
where $(q_1, r_1),\dotsc, (q_M, r_M)$ are arbitrary admissible pairs of exponents.
\end{lemma}
 Furthermore, the solutions to the Schr\"odinger equation enjoy some nice local smoothing properties; there are various results in the literature regarding this property, here we recall a theorem, due to Constatin and Saut \cite{CS} which actually covers a more general setting. In particular it can been shown the solutions of the Schr\"odinger equation are locally more regular than their initial data. Close results have been proved by P. Sj\"olin \cite{Sj} and L. Vega \cite{V}.
 \begin{theorem}[Constantin, Saut \cite{CS}]\label{thm:smooth1}
Let $u$ solves the free Schr\"odinger equation
\begin{equation}
\left\{\begin{array}{l}
i\d_tu+\Delta u=0\\
u(0)=u_0.
\end{array}\right.
\end{equation}
Let $\chi\in\cfun^\infty_0(\R^{1+2})$ of the form
\begin{equation*}
\chi(t,x)=\chi_0(t)\chi_1(x_1)\chi_2(x_2),
\end{equation*}
with $\chi_j\in\cfun^\infty_0(\R)$. Then we have
\begin{equation*}
\int_{\R^{1+2}}\chi^2(t,x)|(I-\Delta)^{1/4}u(t,x)|^2\de x\de t\leq C^2\|u_0\|_{L^2(\R^2)}.
\end{equation*}
In particular, if $u_0\in L^2(\R^2)$, one has for all $T>0$ 
\begin{equation*}
u\in L^2([0,T];H^{1/2}_{loc}(\R^2)).
\end{equation*}
\end{theorem}
We have a similar result also for the nonhomogeneous case:
\begin{theorem}[Constantin, Saut \cite{CS}]\label{thm:smooth2}
Let $u$ be the solution of 
\begin{equation}
\left\{\begin{array}{l}
i\d_tu+\Delta u=F\\
u(0)=u_0\in L^2(\R^2),
\end{array}\right.
\end{equation}
where $F\in L^1([0,T];L^2(\R^2))$. Then it follows
\begin{equation*}
u\in L^2([0,T];H^{1/2}_{loc}(\R^2)).
\end{equation*}
Moreover, let $\chi\in\cfun_0^\infty(\R^{1+2})$ be of the form
\begin{equation*}
\chi(t, x)=\chi_0(t)\chi_1(x_1)\chi_2(x_2),
\end{equation*}
with $\chi_j\in\cfun^\infty_0(\R)$, $\supp\chi_0\subset[0, T]$. Then the following local smoothing estimate holds
\begin{multline}
\pare{\int_{\R^{1+2}}\chi^2(t,x)|(I-\Delta)^{1/4}u(t,x)|^2\de x\de t}^{\mez}\\
\leq C\pare{\|u_0\|_{L^2(\R^2)}+\|F\|_{L^1_tL^2_x([0,T]\times\R^2}}.
\end{multline}
\end{theorem}
This results imply that, the free Schr\"odinger group $U(\cdot)$ fulfills the following inequalities
\begin{gather}
\|U(\cdot)u_0\|_{L^2([0,T];H^{1/2}_{loc})}\lesssim\|u_0\|_{L^2(\R^3)}\label{eq:smooth1}\\
\|\int_0^tU(t-s)F(s)\de s\|_{L^2([0,T];H^{1/2}_{loc}(\R^3))}\lesssim\|F\|_{L^1([0,T];L^2(\R^3))}.
\label{eq:smooth2}
\end{gather}
\subsection{Logarithmic Sobolev inequality and the Gagliardo-Nirenberg inequality in 2D}
As we already pointed out in the introduction, in two dimensions the Poisson electrostatic potential is highly singular, because of its kernel. Hence the need to make the further assumption on the initial datum $\rho_0$, as stated in Theorem \ref{thm:2D_QHD}. A fundamental tool to ensure for the energy to be bounded at all times is the \emph{logarithmic Sobolev inequality}.
\begin{theorem}[Carlen, Loss \cite{CL}]\label{thm:log_Sob}
Let $f$ be a non-negative function in $L^1(\R^2)$ such that $f\log f, f\log(1+|x|^2)\in L^1(\R^2)$. If 
$\int f\de x=M$, then
\begin{multline}
\frac{M}{2}\int_{\R^2}f\log f\de x+
\int_{\R^2\times\R^2}f(x)\log|x-y|f(y)\de x\de y\geq\\
\geq C(M):=\frac{M^2}{2}(1+\log\pi+\log M).
\end{multline}
\end{theorem}
Furthermore in the sequel we use the classical \emph{Gagliardo-Nirenberg inequality} in $\R^2$, namely
\begin{equation}\label{eq:GNS}
\|u\|_{L^p}\lesssim\|u\|_{L^2}^{\frac{2}{p}}\|\nabla u\|_{L^2}^{1-\frac{2}{p}},
\end{equation}
for $2<p<\infty$.
\subsection{Compactness tools}
In conclusion we recall a  compactness result due to Rakotoson, Temam \cite{RT} in the spirit of classical results of Aubin, Lions and Simon which will be needed in Section \ref{sect:apr}.
\begin{theorem}[Rakotoson, Temam \cite{RT}]\label{thm:comp}
Let $(V,\|\cdot\|_{V})$, $(H;\|\cdot\|_{H})$ be two separable Hilbert spaces. Assume that $V\subset H$ with a compact and dense embedding. Consider a sequence $\{u^\eps\}$, converging weakly to a function $u$ in $L^2([0,T];V)$, $T<\infty$. Then $u^\eps$ converges strongly to $u$ in 
$L^2([0,T];H)$, if and only if
\begin{enumerate}
\item $u^\eps(t)$ converges to $u(t)$ weakly in $H$ for a.e. $t$;
\item $\lim_{|E|\to0,E\subset[0,T]}\sup_{\eps>0}\int_E\|u^\eps(t)\|_H^2\de t=0$.
\end{enumerate}
\end{theorem}
\section{Polar decomposition}\label{sect:polar}
In this Section we will recall (see \cite{AM}) the main properties related to the polar factorization of wave functions (detalis will be given in the Appendix). The polar decomposition is a useful tool to factorize a wave function $\psi$ into its amplitude 
$\sqrt{\rho}=|\psi|$ and a unitary factor $\phi$ such that $\psi=\sqrt{\rho}\phi$, that we write the hydrodynamical quantities 
$\nabla\sqrt{\rho}, \Lambda$ without the necessity of defining the velocity $u=J/\rho$ and the phase $S$ as in the WKB formalism.
\newline
Let us consider a wave function $\psi\in L^2(\R^2)$ and define the set
\begin{equation}
P(\psi):=\{\phi\;\textrm{measurable}\;|\;\|\phi\|_{L^\infty(\R^2)}\leq1, \sqrt{\rho}\phi=\psi\;
\textrm{a.e. in}\;\R^2\},
\end{equation}
where $\sqrt{\rho}=|\psi|$. Of course if we consider $\phi\in P(\psi)$, then by the definition of the set $P(\psi)$ it is immediate that $|\phi|=1$ a.e. in $\sqrt{\rho}dx$ and $\phi\in P(\psi)$ is uniquely determined a.e. $\sqrt{\rho}dx$ in $\R^2$.
\newline
Through the polar factorization of the wave function $\psi$ we get the following identity holds
\begin{equation}\label{eq:null-form-hydr}
\hbar^2\re(\nabla\conj{\psi}\otimes\nabla\psi)
=\frac{J\otimes J}{\rho}+\hbar^2\nabla\sqrt{\rho}\otimes\nabla\sqrt{\rho}.
\end{equation}
Indeed by writing the hydrodynamical quantities $\nabla\sqrt{\rho}, \Lambda$ in terms of the associated wave function $\psi$ and its polar factor $\phi$, it follows
\begin{equation}
\nabla\sqrt{\rho}=\re(\conj{\phi}\nabla\psi),\qquad\Lambda=\hbar\im(\conj{\phi}\nabla\psi).
\end{equation}
Hence the equation \eqref{eq:null-form-hydr} holds true since it hides a null structure, namely
\begin{equation*}
\hbar^2\re\curly{(\phi\nabla\conj{\psi})\otimes(\conj{\phi}\nabla\psi)}
=\hbar^2\re(\conj{\phi}\nabla\psi)\otimes\re(\conj{\phi}\nabla\psi)+\hbar^2\im(\conj{\phi}\nabla\psi)
\otimes\im(\conj{\phi}\nabla\psi).
\end{equation*}
Moreover this bilinear structure is $H^1-$stable, namely if $\{\psi_n\}\subset H^1(\R^2)$ is a strongly convergent sequence in $H^1$, $\psi_n\to\psi$, then also their related hydrodynamical quantities converge strongly, $\nabla\sqrt{\rho_n}\to\nabla\sqrt{\rho}, \Lambda_n\to\Lambda$ in $H^1(\R^2)$. This is a quite remarkable property, since for the polar factors we could only have a weak-$\star$ convergence in $L^\infty$ (up to passing to a subsequence), since as we already pointed out they are uniquely defined only in $\sqrt{\rho}dx$. Hence the weak convergence of the quantities $\re(\conj{\phi_n}\nabla\psi_n), \im(\conj{\phi_n}\nabla\psi_n)$ is upgraded to the strong one, thanks to the bilinear structure \eqref{eq:null-form-hydr}. Furthermore, it is worth noting that this structure is fundamental for the stability in the energy space: with the WKB formalism for example this couldn't be achieved without some further regularity assumptions. 
\begin{lemma}[Existence and Stability Lemma]\label{lemma:phase}
Let $\psi\in H^1(\R^2)$, $\sqrt{\rho}:=|\psi|$, then:
\begin{itemize}
\item[(i)] there exists $\phi\in P(\psi)\subset L^\infty(\R^2)$ such that 
$\psi=\sqrt{\rho}\phi$ a.e. in $\R^2$, $\sqrt{\rho}\in H^1(\R^2)$, $\phi$ is unique in $\sqrt{\rho}dx$,
$\nabla\sqrt{\rho}=\re(\conj{\phi}\nabla\psi)$, if we denote $\Lambda:=\hbar\im(\conj{\phi}\nabla\psi)$,  then $\Lambda\in L^2(\R^2)$ and moreover
\begin{equation}\label{eq:null-form}
\hbar^2\re(\d_j\conj{\psi}\d_k\psi)=\hbar^2\d_j\sqrt{\rho}\d_k\sqrt{\rho}+\Lambda^{(j)}\Lambda^{(k)}
\end{equation}
and
\begin{equation}\label{eq:119}
\hbar\nabla\conj{\psi}\wedge\nabla\psi=2i\nabla\sqrt{\rho}\wedge\Lambda.
\end{equation}
\item[(ii)] (Stability) Let $\psi_n\to\psi$ strongly in $H^1(\R^2)$, denote by $\phi_n$ the polar factor of $\psi_n$, then  it follows
\begin{equation}\label{eq:strong_conv}
\nabla\sqrt{\rho_n}\to\nabla\sqrt{\rho},\quad \Lambda_n\to\Lambda\qquad\textrm{in}\; L^2(\R^2),
\end{equation}
where $\Lambda_n:=\hbar\im(\conj{\phi_n}\nabla\psi_n)$.
\end{itemize}
\end{lemma}
The previous identity \eqref{eq:119} is needed to obtain the generalized irrotationality condition \eqref{eq:irrot}, since  $\nabla\wedge J=\hbar
\im(\nabla\conj{\psi}\wedge\nabla\psi)$. 
\newline
In the Bohmian mechanics (see \cite{TT}) and in the Nelson's stochastic mechanics (see \cite{Carl}) the major difficulty regards the definition of current and osmotic velocities in the nodal regions ($\rho=0$). The main advantage of our approach is to build up a self consistent theory, based only on particle density and current density, which, thanks to the polar decomposition,  does not require the use of velocity fields.
\section{QHD without collisions and nonlinear Schr\"odinger-Poisson system in 2D}\label{sect:noncoll}
In this section we consider the collisionless case, namely $f=0$,
\begin{equation}\label{eq:QHD_farl}
\left\{\begin{array}{ll}
\d_t\rho+\diver J=0\\
\d_tJ+\diver\pare{\frac{J\otimes J}{\rho}}+\nabla P(\rho)+\rho\nabla V
=\frac{\hbar^2}{2}\rho\nabla\pare{\frac{\Delta\sqrt{\rho}}{\sqrt{\rho}}}\\
-\Delta V=\rho,
\end{array}\right.
\end{equation}
and our goal is to show the connection with the nonlinear Schr\"odinger-Poisson system
\begin{equation}\label{eq:2D_NLS}
\left\{\begin{array}{l}
i\hbar\d_t\psi=-\frac{\hbar^2}{2}\Delta\psi+|\psi|^{p-1}\psi+V\psi\\
-\Delta V=|\psi|^2,
\end{array}\right.
\end{equation}
with initial datum $\psi(0)=\psi_0$.
\newline
Indeed, let us consider a (global) solution to \eqref{eq:2D_NLS}, and the equation for the mass density $\rho=|\psi|^2$ is immediately given by
\begin{equation*}
\d_t\rho+\diver J=0,
\end{equation*}
where $J=\hbar\im(\conj{\psi}\nabla\psi)$. Now, if we compute the time derivative of the current density $J$, we find out the following conservation law
\begin{equation*}
\d_tJ+\diver\pare{\hbar^2\re(\nabla\conj{\psi}\otimes\nabla\psi)}-\frac{\hbar^2}{4}\Delta\nabla\rho+\nabla P(\rho)+\rho\nabla V=0,
\end{equation*}
where $V$ is the Poisson electrostatic potential given by $-\Delta V=\rho$, and 
$P(\rho)=\frac{p-1}{p+1}\rho^{(p+1)/2}$. 
Now, from formula \eqref{eq:null-form} we can write
\begin{equation*}
\hbar^2\re(\nabla\conj{\psi}\otimes\nabla\psi)=\hbar^2\nabla\sqrt{\rho}\otimes\nabla\sqrt{\rho}+
\Lambda\otimes\Lambda,
\end{equation*}
where $\Lambda$ is defined as in Lemma \ref{lemma:phase}. Finally, we just notice that we can write, as in equation \eqref{eq:bohm}
\begin{equation*}
\frac{1}{4}\Delta\nabla\rho-\diver(\nabla\sqrt{\rho}\otimes\nabla\sqrt{\rho})=\frac{1}{2}
\rho\nabla\pare{\frac{\Delta\sqrt{\rho}}{\sqrt{\rho}}},
\end{equation*}
where the distribution $T=\frac{1}{2}\rho\nabla\pare{\frac{\Delta\sqrt{\rho}}{\sqrt{\rho}}}$ is defined by
\begin{equation*}
\scalar{T}{\phi}=\int_{\R^2}\nabla\sqrt{\rho}\otimes\nabla\sqrt{\rho}:\nabla\phi-\frac{1}{4}\rho\Delta\diver\phi dx,\quad\textrm{for all }\phi\in\cfun^\infty_0(\R^2),
\end{equation*}
so that now the conservation law for the current density read
\begin{equation}
\d_tJ+\diver\pare{\frac{J\otimes J}{\rho}}+\nabla P(\rho)+\rho\nabla V
=\frac{\hbar^2}{2}\rho\nabla\pare{\frac{\Delta\sqrt{\rho}}{\sqrt{\rho}}}.
\end{equation}
Hence, given a solution $\psi$ to \eqref{eq:2D_NLS}, with initial datum $\psi(0)=\psi_0$, by defining 
$(\rho, J):=(|\psi|^2, \hbar\im(\conj{\psi}\nabla\psi))$ we have a solution to \eqref{eq:QHD_farl} with initial data $(\rho_0, J_0)=(|\psi_0|^2, \hbar\im(\conj{\psi_0}\nabla\psi_0))$. 
These arguments can be written in a rigorous way by using a regularization argument as in Proposition 15 in \cite{AM}.
\begin{proposition}\label{prop:19}
Let $\psi_0\in H^1(\R^2)$, define the initial data, 
$(\rho_0, J_0):=(|\psi_0|^2, \hbar\im(\conj{\psi_0}\nabla\psi_0))$, moreover assume the following conditions hold:
\begin{equation}\label{eq:83}
\int_{\R^2}\rho_0\log\rho_0\de x<\infty,
\end{equation}
and
\begin{equation}\label{eq:84}
V(0, \cdot):=-\frac{1}{2\pi}\int_{\R^2}\log|\cdot-y|\rho_0(y)\de y\in L^r(\R^2),
\end{equation}
for some $2<r<\infty$.
Then for any $T>0$ there exists a FEWS $(\rho, J)$ to the Cauchy problem \eqref{eq:QHD_farl} in the space-time slab $[0, T)\times\R^2$. Furthermore the energy $E(t)$ defined as in \eqref{eq:en} is conserved for all times $t\in[0, T)$.
\end{proposition}
Consequently, to prove the Proposition \ref{prop:19} we only need to state a global well-posedness result for the Cauchy problem associated to equation \eqref{eq:2D_NLS}. Indeed, let us point out that conditions \eqref{eq:83} and \eqref{eq:84} are needed to provide a global solution to the Cauchy problem associated to equation \eqref{eq:2D_NLS}. 
Furthermore, as we will see in the next section, we also need the solution to \eqref{eq:2D_NLS} to satisfy appropriate estimates in some space-time mixed Lebesgue spaces. The result we prove is the following
\begin{theorem}\label{thm:2D_NLS}
Let us consider the Cauchy problem \eqref{eq:2D_NLS} with initial datum
\begin{equation}\label{eq:2D_NLS_iv}
\psi(0)=\psi_0,
\end{equation}
such that
\begin{equation}\label{eq:85}
\psi_0\in H^1(\R^2),\qquad\int_{\R^2}|\psi_0|^2\log|\psi_0|^2\de x<\infty,
\end{equation}
and
\begin{equation}\label{eq:86}
V(0, \cdot):=-\frac{1}{2\pi}\int_{\R^2}\log|\cdot-y|\rho_0(y)\de y\in L^r(\R^2),
\end{equation}
for some $2<r<\infty$. Then, there exists a unique solution 
$\psi\in\cfun([0, \infty);H^1(\R^2))$ to \eqref{eq:2D_NLS}, \eqref{eq:2D_NLS_iv}. Furthermore
the energy
\begin{multline}\label{eq:2D_en}
E(t):=\int_{\R^2}\frac{\hbar^2}{2}|\nabla\psi|^2+\frac{2}{p+1}|\psi|^{p+1}
+\mez V|\psi|^2\de x\\
=\int_{\R^2}\frac{\hbar^2}{2}|\nabla\psi(t, x)|^2+\frac{2}{p+1}|\psi(t, x)|^{p+1}\de x\\
\quad-\frac{1}{4\pi}\int_{\R^2\times\R^2}\rho(t, x)\log|x-y|\rho(t, y)\de x\de y,
\end{multline}
is well-defined for all times and it is conserved. Finally, 
for every $T>0$ and for every $(q, r)$ admissible pair of Schr\"odinger exponents in $\R^2$, we have
\begin{equation}\label{eq:2D_Strich}
\|\nabla\psi\|_{L^q_tL^r_x([0, T)\times\R^2)}\leq C(T, \|\psi_0\|_{L^2}, E_0).
\end{equation}
\end{theorem}
\begin{proof}
Let us remark that now we're going to deal with the case in which the power-like nonlinearity $|\psi|^{p-1}\psi$ satisfies $1<p<\infty$; the case $p=1$ is straightforward. 
\newline
Firstly, from the logarithimic Sobolev inequality (see theorem \ref{thm:log_Sob}), we infer the boundedness of the initial energy, and furthermore, since it is conserved for all times, we have
\begin{equation*}
E(t)=E_0\lesssim\|\psi_0\|_{L^2}^2+\|\nabla\psi_0\|_{L^2}^2+\int\rho_0\log\rho_0\de x<\infty.
\end{equation*}
Secondly, since we have $\psi\in H^1(\R^2)$, from the Gagliardo-Nirenberg-Sobolev inequality \eqref{eq:GNS} we can say that $\|\psi\|_{L^p(\R^2)}$ for each 
$2\leq p<\infty$, and consequently $\rho\in L^p(\R^2)$ for each $1\leq p<\infty$.
\newline
Now, we want to give a local well-posedness result for the solutions to \eqref{eq:2D_NLS}. From standard arguments (see \cite{T}), we look for \emph{a priori} estimates of solutions to 
\eqref{eq:2D_NLS} in the Strichartz spaces. 
To this aim, we first need to know the mixed space-time Lebesgue spaces in which  the electrostatic potential lies. 
\begin{lemma}\label{lemma:elec_pot}
Let $(\psi, V)$ satisfy the Schr\"odinger-Poisson system \eqref{eq:2D_NLS} in $[0, T)\times\R^2$ for some $T>0$. Let us further assume that $V(0)\in L^r(\R^2)$, for some $2<p<\infty$. Then
\begin{equation}\label{eq:elec_pot_reg}
\left\{\begin{array}{l}
\nabla V\in L^\infty_tL^p_x([0, T)\times\R^2),\qquad\forall\;p\in[2, \infty),\\
V\in L^\infty_tL^r_x([0, T)\times\R^2).
\end{array}\right.
\end{equation}
\end{lemma}
\begin{remark}
Let us note that the first \emph{a priori} estimates in \eqref{eq:elec_pot_reg} for the gradient of the electrostatic potential don't imply any estimates on the potential $V$, since we can't use the Sobolev embedding theorem. Hence we need some additional assumptions on $V(0)$, namely hypothesis \eqref{eq:86}.
\end{remark}
\begin{proof}
Clearly by the elliptic regularity theory and by the Sobolev embedding theorem we have
\begin{equation}
\|\nabla V\|_{L^\infty L^2}\leq\||\nabla|^{-1}\rho\|_{L^\infty L^2}\lesssim\|\rho_0\|_{L^1}.
\end{equation}
Furthermore, by taking the gradient of \eqref{eq:elec_pot} we get
\begin{equation}
\nabla V(t, x)=-\frac{1}{2\pi}\int_{\R^2}\frac{x-y}{|x-y|^2}\rho(t, y)\de y,
\end{equation}
and by using the generalized Hardy-Littlewood-Sobolev's inequality we obtain that 
$\nabla V\in L^\infty([0, T); L^p(\R^2)), 2<p<\infty$. Resuming, we infer 
$\nabla V\in L^\infty([0, T); L^p(\R^2))$, for $2\leq p<\infty$.
\newline
Let us now differentiate with respect to time the formula \eqref{eq:elec_pot}, we thus obtain, by using the continuity equation,
\begin{equation*}
\d_tV(t, x)=-\frac{1}{2\pi}\int_{\R^2}\frac{x-y}{|x-y|^2}\cdot J(t, y)\de y,
\end{equation*}
and consequently
\begin{equation}
V(t, x)=V(0, x)-\frac{1}{2\pi}\int_0^t\int_{\R^2}\frac{x-y}{|x-y|^2}\cdot J(s, y)\de y\de s.
\end{equation}
Now, from the bounds on the energy and the Gagliardo-Nirenberg-Sobolev inequality we infer $J\in L^\infty([0, \infty); L^p(\R^2))$, for $1\leq p<2$, thus the generalized Hardy-Littlewood-Sobolev inequality yields $\d_tV\in L^\infty([0, \infty); L^r(\R^2))$, with 
$2<r<\infty$. Hence
\begin{equation}
\|V(t)\|_{L^r}\leq\|V(0)\|_{L^r}+t\|\d_tV\|_{L^\infty_tL^r_x([0, \infty)\times\R^2)},\qquad 2<r<\infty,
\end{equation}
and consequently $V\in L^\infty_tL^r([0, T)\times\R^2)$.
\end{proof}
Let us now come back now to the proof of theorem \ref{thm:2D_NLS}. By Strichartz estimates we have
\begin{multline*}
\|\nabla\psi\|_{L^q_tL^r_x([0, T)\times\R^2)}\lesssim\|\nabla\psi_0\|_{L^2(\R^2)}
+\||\psi|^{p-1}\nabla\psi\|_{L^{q_1'}_tL^{r_1'}_x([0, T)\times\R^2)}\\
+\|V\nabla\psi\|_{L^{q_2'}_tL^{r_2'}_x([0, T)\times\R^2)}
+\|\nabla V\psi\|_{L^{q_3'}_tL^{r_3'}_x([0, T)\times\R^2)},
\end{multline*}
where $(q_i', r_i')$ are dual exponents of an admissible pair for Schr\"odinger in $\R^2$, $i=1, 2, 3$.
Let us consider the second term in the right hand side, $\lplqs{|\psi|^{p-1}\nabla\psi}{q_1'}{r_1'}$. By H\"older's inequality we have
\begin{equation*}
\||\psi|^{p-1}\nabla\psi\|_{L^{q_1'}_tL^{r_1'}_x}\leq T^{1/q_1'}
\|\psi\|_{L^\infty_tL^{r(p-1)}}^{p-1}
\|\nabla\psi\|_{L^\infty_tL^2_x},
\end{equation*}
where we have the following algebraic conditions on the exponents $r_1', r$
\begin{equation*}
\left\{\begin{array}{l}
0<\1{r}\leq\frac{p-1}{2}\\
\mez\leq\1{r_1'}=\1{r}+\mez<1
\end{array}\right.
\end{equation*}
Then, by applying the Gagliardo-Nirenberg \eqref{eq:GNS} inequality we get
\begin{multline*}
\lplqs{|\psi|^{p-1}\nabla\psi}{q_1'}{r_1'}\leq T^{1/q_1'}\lplqs{\psi}{\infty}{r(p-1)}^{p-1}
\lplqs{\nabla\psi}{\infty}{2}\\
\lesssim T^{1/q_1'}\lplqs{\psi}{\infty}{2}^{\frac{2}{r}}\lplqs{\nabla\psi}{\infty}{2}^{\frac{r(p-1)-2}{r(p-1)}(p-1)}
\lplqs{\nabla\psi}{\infty}{2}
=T^{1/q_1'}\lplqs{\psi}{\infty}{2}^{\frac{2}{r}}\lplqs{\nabla\psi}{\infty}{2}^{\frac{rp-2}{r}}.
\end{multline*}
For the term $\lplqs{V\nabla\psi}{q_2'}{r_2'}$ we choose $(q_2', r_2')=(r', \frac{2r}{r+2})$, where $r$ is the exponent such that $V(0)\in L^r$, and hence
\begin{multline*}
\lplqs{V\nabla\psi}{r'}{\frac{2r}{r+2}}\leq T^{\1{r'}}\lplqs{V}{\infty}{r}\lplqs{\nabla\psi}{\infty}{2}\\
\lesssim T^{\1{r'}}(\|V(0)\|_{L^r}+T\|\d_tV\|_{L^\infty_tﬂ^r})\lplqs{\nabla\psi}{\infty}{2}.
\end{multline*}
Finally for the last term $\lplqs{\nabla V\psi}{q_3'}{r_3'}$ we can easily choose three exponents 
$r_3', p, p_1$, such that
\begin{equation*}
\lplqs{\nabla V\psi}{q_3'}{r_3'}\leq T^{\1{q_3'}}\lplqs{\nabla V}{\infty}{p}\lplqs{\psi}{\infty}{p_1}
\lesssim T^{\1{q_3'}}\lplqs{\nabla V}{\infty}{p}\lplqs{\psi}{\infty}{2}^{\frac{2}{p_1}}
\lplqs{\nabla\psi}{\infty}{2}^{1-\frac{2}{p_1}},
\end{equation*}
and
\begin{equation*}
\left\{\begin{array}{l}
0<\1{p}\leq\mez,\qquad0<\1{p_1}\leq\mez\\
\mez\leq\1{r_3'}=\1{p}+\1{p_1}<1.
\end{array}\right.
\end{equation*}
Hence we obtain the following estimate
\begin{equation}
\lplq{\nabla\psi}{q}{r}{[0, T)}\lesssim\|\nabla\psi_0\|_{L^2}+T^\alpha C(\|\psi_0\|_{L^2}, E_0),
\end{equation}
for some $\alpha>0$. 
Thus by a standard argument we can state there exists a 
$T^\ast=T^\ast(\|\psi_0\|_{L^2}, \|\nabla\psi_0\|_{L^2})$, such that there exists a unique solution
\begin{equation*}
\psi\in\cfun([0, T^\ast); H^1(\R^2))\cap L^q([0, T^\ast); L^r(\R^2))\qquad\forall\;(q, r)
\;\textrm{adimssible pair},
\end{equation*}
of the Cauchy problem for \eqref{eq:2D_NLS}. 
\newline
Furthermore, again by standard arguments, namely since the energy is conserved for every time, we can repeat the same argument for each space-time slab $[kT^\ast, (k+1)T^\ast]$, since $T^\ast$ depends only on the initial datum, and thus we can extend this solution globally. Hence theorem 
\ref{thm:2D_NLS} is proved.
\end{proof}
\section{Collisional case: the fractional step approximations}\label{sect:fract}
Following the approach in \cite{AM}, we want to use the tools presented in the previous sections  in order to construct a \emph{sequence of approximate solutions} to system \eqref{eq:QHD}.
\begin{definition}
 $\{(\rho^\tau, J^\tau)\}$ is a \emph{sequence of approximate solutions} to the system \eqref{eq:QHD}, with locally integrable initial data $(\rho_0, J_0)$, if 
$(\sqrt{\rho^\tau}, \Lambda^\tau)\in L^2_{loc}([0, T);H^1_{loc}(\R^2))\times L^2_{loc}([0, T);L^2_{loc}(\R^2))$, $\rho^\tau:=(\sqrt{\rho^\tau})^2$, $J^\tau:=\sqrt{\rho^\tau}\Lambda^\tau$, satisfy
\begin{itemize}
\item for any test function $\eta\in\cfun_0^\infty([0,T)\times\R^2)$
 \begin{equation}\label{eq:approx_QHD1}
 \int_0^T\int_{\R^2}\rho^\tau\d_t\eta+J^\tau\cdot\nabla\eta\de x\de t
 +\int_{\R^2}\rho_0\eta(0)\de x=o(1)
 \end{equation}
 as $\tau\to0$,
 \item for any test function $\zeta\in\cfun_0^\infty([0,T)\times\R^2;\R^2)$
 \begin{multline}\label{eq:approx_QHD2}
 \int_0^T\int_{\R^2}J^\tau\cdot\d_t\zeta+\Lambda^\tau\otimes\Lambda^\tau:\nabla\zeta
 +P(\rho^\tau)\diver\zeta
 -\rho^\tau\nabla V^\tau\cdot\zeta-J^\tau\cdot\zeta\\
 +\hbar^2\nabla\sqrt{\rho^\tau}\otimes\nabla\sqrt{\rho^\tau}:\nabla\zeta
 -\frac{\hbar^2}{4}\rho^\tau\Delta\diver\zeta\de x\de t+\int_{\R^2}J_0\cdot\zeta(0)\de x=o(1)
 \end{multline}
 as $\tau\to0$;
 \item \emph{(generalized irrotationality condition)} in the sense of distribution
 \begin{equation*}
 \nabla\wedge J^\tau=2\nabla\sqrt{\rho^\tau}\wedge\Lambda^\tau.
 \end{equation*}
 \end{itemize}
\end{definition}
Our fractional step method is based on an operator splitting method which composes the solution operator of the non-collisional problem with the approximate solution operator (semi-discrete explicit Euler scheme) of the collisional problem.  Given a (small) parameter $\tau>0$, we solve a non-collisional QHD problem as in Section \ref{sect:noncoll}, hence we solve the collisional problem without QHD. At this point in order to restart with the non-collisional QHD problem, the main difficulty which needs to be solved regards how to update the initial data at each time step. Indeed, as remarked in the previous section we are able to solve the non-collisional QHD only in the case of Cauchy data compatible with the Schr\"odinger picture. This fact imposes the necessity to reconstruct a wave function at each time step, which is made possible by the polar factorization of the wave function computed in the previous time step.
\newline
As previously remarked, this method requires special type of initial data $(\rho_0, J_0)$, namely we assume there exists $\psi_0\in H^1(\R^3)$, such that the hydrodynamic initial data are given via the Madelung transforms
\begin{equation*}
\rho_0=|\psi_0|^2,\qquad J_0=\hbar\im(\conj{\psi_0}\nabla\psi_0).
\end{equation*}
This assumption is physically relevant since it implies the compatibility of our solutions to the QHD problem with the wave mechanics approach. The iteration procedure can be defined in this following way. First of all, we take $\tau>0$, which will be the time mesh unit; therefore we define the approximate solutions in each time interval $[k\tau, (k+1)\tau)$, for any integer $k\geq0$.
\newline
At the first step, $k=0$, we solve the Cauchy problem for the Schr\"odinger-Poisson system
\begin{equation}\label{eq:NLS2}
\left\{\begin{array}{l}
i\hbar\d_t\psi^\tau+\frac{\hbar^2}{2}\Delta\psi^\tau=|\psi^\tau|^{p-1}\psi^\tau+V^\tau\psi^\tau\\
-\Delta V^\tau=|\psi^\tau|^2\\
\psi^\tau(0)=\psi_0
\end{array}\right.
\end{equation}
by looking for the restriction in $[0, \tau)$ of the unique strong solution 
$\psi\in\cfun(\R; H^1(\R^2))$ (see Section \ref{sect:noncoll}).
Let us define 
$\rho^\tau:=|\psi^\tau|^2, J^\tau:=\hbar\im(\conj{\psi^\tau}\nabla\psi^\tau)$, then, as shown in the previous section, $(\rho^\tau, J^\tau)$ is a weak solution to the non-collisional QHD system. 
Let us assume that we know $\psi^\tau$ in the space-time slab 
$[(k-1)\tau, k\tau)\times\R^2$, we want to set up a recursive method, hence we have to show how to define $\psi^\tau, \rho^\tau, J^\tau$ in the strip $[k\tau, (k+1)\tau)$.
\newline
In order to take into account the presence of the collisional term $f=J$ we update $\psi^\tau$ in $t=k\tau$, namely we define $\psi^\tau(k\tau+)$, in such a way that its hydrodynamical quantities satisfy
\begin{equation}\label{eq:87}
\rho^\tau(k\tau+)=\rho^\tau(k\tau-),\qquad J^\tau(k\tau+)=(1-\tau)J^\tau(k\tau-),
\end{equation}
which means, the hydrodynamical quantities approximately solve (as a first order approximation) the collisional problem without QHD. This could be done by means of the polar decomposition, described in Section \ref{sect:polar}. Indeed, denote by
\begin{equation*}
\psi^\tau(k\tau-)=\phi^\tau_{k}\sqrt{\rho^\tau}(k\tau-),
\end{equation*}
where clearly $\phi^\tau_k$ is a polar factor for the wave function $\psi^\tau(k\tau-)$. Hence a way of defining $\psi^\tau(k\tau+)$ might be the following one
\begin{equation}\label{eq:88}
\psi^\tau(k\tau+)=(\phi^\tau_k)^{(1-\tau)}\sqrt{\rho^\tau}(k\tau-),
\end{equation}
in this way the equations \eqref{eq:87} would be exactly fulfilled, but
unfortunately, as we will see in the next Section, this way of updating the wave function $\psi^\tau$ presents some problems, due to the lack of differentiability of the polar factor $\phi$, when we try to get \emph{a priori} bounds for the sequence of approximate solutions. In order to circumvent this difficulty we will use approximations for the updating in \eqref{eq:88} which are stable in the energy norm, namely we will apply Lemma \ref{lemma:updat} in the Appendix, with $\psi=\psi^\tau(k\tau-)$, 
$\eps=\tau2^{-k}\|\psi_0\|_{H^1(\R^2)}$. In this way we can define
\begin{equation}\label{eq:125}
\psi^\tau(k\tau+)=\tilde\psi,
\end{equation}
by using the wave function $\tilde\psi$ defined in the Lemma \ref{lemma:updat}. Therefore we have
\begin{align}
\rho^\tau(k\tau+)&=\rho^\tau(k\tau-)\\
\Lambda^\tau(k\tau+)&=(1-\tau)\Lambda^\tau(k\tau-)+R_k\label{eq:126},
\end{align}
where $\|R_k\|_{L^2(\R^2)}\leq\tau2^{-k}\|\psi_0\|_{H^1(\R^2)}$ and
\begin{equation}\label{eq:127}
\nabla\psi^\tau(k\tau+)=\nabla\psi^\tau(k\tau-)-i\frac{\tau}{\hbar}\phi^\ast\Lambda^\tau(k\tau-)+r_{k, \tau},
\end{equation}
with $\|\phi^\ast\|_{L^\infty}\leq1$ and
\begin{equation*}
 \|r_{k, \tau}\|_{L^2}\leq C(\tau\|\nabla\psi^\tau(k\tau-)\|+\tau2^{-k}\|\psi_0\|_{H^1(\R^2)})
 \lesssim\tau E_0^{\mez}.
 \end{equation*}
 Let us remark that \eqref{eq:127} will play a fundamental role to prove Strichartz-type estimates for the approximating sequence.
\newline
We then solve again the Schr\"odinger-Poisson system with initial data $\psi(0)=\psi^\tau(k\tau+)$. We define $\psi^\tau$ in the time strip $[k\tau, (k+1)\tau)$ as the restriction of the Schr\"odinger-Poisson solution just found in $[0, \tau)$, furthermore, we define $\rho^\tau:=|\psi^\tau|^2, 
J^\tau:=\hbar\im(\conj{\psi^\tau}\nabla\psi^\tau)$ as the solution of the non-collisional QHD 
\eqref{eq:QHD_farl}.
\newline
With this procedure  we can go on every time strip and then construct an approximate solution 
$(\rho^\tau, J^\tau, V^\tau)$ of the QHD system.
\begin{theorem}[Consistency of the approximate solutions]\label{thm:cons}
Let us consider the sequence of approximate solutions $\{(\rho^\tau, J^\tau)\}_{\tau>0}$ constructed via the fractional step method. Assume there exists 
$(\sqrt{\rho}, \Lambda)\in L^2_{loc}([0, T);H^1_{loc}(\R^2))\times L^2_{loc}([0, T);L^2_{loc}(\R^2))$, such that
\begin{align}
\sqrt{\rho^\tau}&\to\sqrt{\rho}\qquad\textrm{in}\;L^2([0, T);H^1_{loc}(\R^2))\\
\Lambda^\tau&\to\Lambda\qquad\textrm{in}\;L^2([0, T); L^2_{loc}(\R^2)).
\end{align}
Then the limit function $(\rho, J)$, where as before $J=\sqrt{\rho}\Lambda$, is a FEWS to the QHD system, with Cauchy data 
$(\rho_0, J_0)$.
\end{theorem}
\begin{proof}
Let us plug the approximate solutions $(\rho^\tau, J^\tau)$ in the weak formulation and let 
$\zeta\in\cfun_0^\infty([0, T)\times\R^2)$, then we have
\begin{gather*}
\begin{split}
\int_0^\infty\int_{\R^2}J^\tau\cdot\d_t\zeta+\Lambda^\tau\otimes\Lambda^\tau:\nabla\zeta
+P(\rho^\tau)\diver\zeta-\rho^\tau\nabla V^\tau\cdot\zeta-J^\tau\cdot\zeta\\
+\hbar^2\nabla\sqrt{\rho^\tau}\otimes\nabla\sqrt{\rho^\tau}:\nabla\zeta
-\frac{\hbar^2}{4}\rho^\tau\Delta\diver\zeta\de x\de t+\int_{\R^2}J_0\cdot\zeta(0)\de x
\end{split}\\
\begin{split}
=\sum_{k=0}^\infty\slabint J^\tau\cdot\d_t\zeta+\Lambda^\tau\otimes\Lambda^\tau:\nabla\zeta
+P(\rho^\tau)\diver\zeta-\rho^\tau\nabla V^\tau\cdot\zeta-J^\tau\cdot\zeta\\
+\hbar^2\nabla\sqrt{\rho^\tau}\otimes\nabla\sqrt{\rho^\tau}:\nabla\zeta
-\frac{\hbar^2}{4}\rho^\tau\Delta\diver\zeta\de x\de t
+\int_{\R^2}J_0\cdot\zeta(0)\de x
\end{split}\\
\begin{split}
=\sum_{k=0}^\infty\bigg[\slabint-J^\tau\cdot\zeta\de x\de t
+\int_{\R^2}J^\tau((k+1)\tau-)\cdot\zeta((k+1)\tau)\\
-J^\tau(k\tau)\cdot\zeta(k\tau)\de x\bigg]
+\int_{\R^2}J_0\cdot\zeta(0)\de x
\end{split}\\
\end{gather*}
\begin{gather*}
\begin{split}
=\sum_{k=0}^\infty\slabint-J^\tau\cdot\zeta\de x\de t
+\sum_{k=1}^\infty\int_{\R^2}\pare{J^\tau(k\tau-)-J^\tau(k\tau)}\cdot\zeta(k\tau)\de x
\end{split}\\
\begin{split}
=\sum_{k=0}^\infty\slabint-J^\tau\cdot\zeta\de x\de t
+\sum_{k=1}^\infty\tau\int_{\R^2}J^\tau(k\tau-)\cdot\zeta(k\tau)\de x\\
+\sum_{k=1}^\infty\int_{\R^2}(1-\tau)\sqrt{\rho^\tau(k\tau)}R_k\cdot\zeta(k\tau)\de x
\end{split}\\
\begin{split}
=\sum_{k=0}^\infty\slabint J^\tau((k+1)\tau-)\cdot\zeta((k+1)\tau)-J^\tau(t)\cdot\zeta(t)\de x\de t\\
+(1-\tau)\sum_{k=0}^\infty\int_{\R^2}\sqrt{\rho^\tau}(k\tau)R_k\cdot\zeta(k\tau)\de x
\end{split}\\
=o(1)+O(\tau)
\end{gather*}
as $\tau\to0$.
\end{proof}
\section{a priori estimates and convergence}\label{sect:apr}
In this section we obtain various a priori estimates needed to show the compactness of the sequence of approximate solutions $(\rho^\tau, J^\tau)$ in some appropriate function spaces. As we stated in the Theorem \ref{thm:cons}, we wish to prove the strong convergence of $\{\sqrt{\rho^\tau}\}$ in $L^2_{loc}([0, T);H^1_{loc}(\R^2))$ and 
$\{\Lambda^\tau\}$ in $L^2_{loc}([0, T);L^2_{loc}(\R^2))$. To achieve this goal we use a compactness result in the class of the Aubin-Lions's type lemma, due to Rakotoson-Temam \cite{RT} (see the Section 2). The plan of this section is the following, first of all we get a discrete version of the (dissipative) energy inequality for the system \eqref{eq:QHD}, then we use the Strichartz estimates for $\nabla\psi^\tau$ by means of the formula \eqref{eq:131} below. Consequently via the Strichartz estimates and by using the local smoothing results of Theorems \ref{thm:smooth1}, \ref{thm:smooth2}, we deduce some further regularity properties on the sequence $\{\nabla\psi^\tau\}$. This fact will be stated in the next Proposition 
\ref{prop:28}. In this way it is possible to get the regularity properties of the sequence $\{\nabla\psi^\tau\}$ which are needed to apply the Theorem \ref{thm:comp} and hence to get the convergence for $\{\sqrt{\rho^\tau}\}$ and $\{\Lambda^\tau\}$.
\newline
Concerning the energy inequality we notice that for a sufficiently regular solution to the QHD system one has
\begin{equation}\label{eq:en_diss}
E(t)=-\int_0^t\int_{\R^2}|\Lambda|^2\de x\de t'+E_0,
\end{equation}
where the energy is defined as in \eqref{eq:en}.
Now we would like to find a discrete version of the energy dissipation \eqref{eq:en_diss} for the approximate solutions to \eqref{eq:QHD}. 
\begin{proposition}[Discrete energy inequality]
Let $(\rho^\tau, J^\tau)$ be an approximate solution to the QHD system, with $0<\tau<1$. Then, for all 
$t\in[N\tau, (N+1)\tau)$ it follows
\begin{equation}\label{eq:en-diss}
E^\tau(t)\leq-\frac{\tau}{2}\sum_{k=1}^N\|\Lambda(k\tau-)\|_{L^2(\R^2)}+(1+\tau)E_0.
\end{equation}
\end{proposition}
\begin{proof}
For all $k\geq1$, we have
\begin{align*}
E^\tau(k\tau+)-E^\tau(k\tau-)=&
\int\mez|\Lambda^\tau(k\tau+)|^2-\mez|\Lambda^\tau(k\tau-)|^2\de x\\
=&\mez\int(-2\tau+\tau^2)|\Lambda^\tau(k\tau-)|^2\\
&\qquad+2(1-\tau)\Lambda^\tau(k\tau-)\cdot R_k+|R_k|^2\de x\\
\leq&\mez\int(-2\tau+\tau^2)|\Lambda^\tau(k\tau-)|^2\de x+(1-\tau)\alpha|\Lambda^\tau(k\tau-)|^2\\
&\qquad+\frac{1-\tau}{\alpha}|R_k|^2+|R_k|^2\de x\\
=&\mez\int(-2\tau+\tau^2+\alpha-\alpha\tau)|\Lambda^\tau(k\tau-)|^2\\
&\qquad+\pare{\frac{1-\tau+\alpha}{\alpha}}|R_k|^2\de x.
\end{align*}
Here $R_k$ denotes the error term as in espression \eqref{eq:126}.
If we choose $\alpha=\tau$, it follows
\begin{multline}\label{eq:124}
E^\tau(k\tau+)-E^\tau(k\tau-)\leq-\frac{\tau}{2}\|\Lambda^\tau(k\tau-)\|_{L^2}^2
+\1{2\tau}\|R_k\|^2_{L^2}\\
\leq-\frac{\tau}{2}\|\Lambda^\tau(k\tau-)\|^2_{L^2}+\tau2^{-k-1}\|\psi_0\|_{H^1}.
\end{multline}
The inequality \eqref{eq:en-diss} follows by summing up all the terms in \eqref{eq:124} and by the energy conservation in each time strip $[k\tau, (k+1)\tau)$.
\end{proof}
Unfortunately the energy estimates are not sufficient to get the necessary compactness to show the convergence of the sequence of approximate solutions. Indeed from the discrete energy inequality, we only get the weak convergence of $\nabla\psi^\tau$ in $L^\infty([0, \infty);H^1(\R^2))$, and therefore the quadratic terms in \eqref{eq:approx_QHD2} could exhibit some concentrations phenomena in the limit.
\newline
More precisely, 
from energy inequality we get the sequence $\{\psi^\tau\}$ is uniformly bounded in 
$L^\infty([0,\infty);H^1(\R^2))$, hence there exists $\psi\in L^\infty([0, \infty);H^1(\R^2))$, such that
\begin{equation*}
\psi^\tau\wlim\psi\qquad\textrm{in}\;L^\infty([0,\infty);H^1(\R^2)).
\end{equation*}
Therefore we get
\begin{align*}
\sqrt{\rho^\tau}&\wlim\sqrt{\rho}\qquad\textrm{in}\;L^\infty([0,\infty);H^1(\R^2))\\
\Lambda^\tau&\wlim\Lambda\qquad\textrm{in}\;L^\infty([0,\infty);L^2(\R^2))\\
J^\tau&\wlim\sqrt{\rho}\Lambda\qquad\textrm{in}\;L^\infty([0, \infty);L^1_{loc}(\R^2)).
\end{align*}
The need to pass into the limit the quadratic expressions leads us to look for a priori estimates in stronger norms. The link with Schr\"odinger equation brings naturally into this search the Strichartz-type estimates and the following results are concerned with these estimates. However they are \emph{not} an immediate consequence of the Strichartz estimates for the Schr\"odinger equation since the wave function obtained in the updating procedure implement in the fractional step in general does not solve any Sch\"odinger equation.
\newline
The following Remark \ref{rmk:25} will clarify in detail why we used the result in Lemma \ref{lemma:updat} to update the wave function in place of using the exact factorization property as in equation \eqref{eq:88}.
\begin{lemma}\label{lemma:28}
Let $\psi^\tau$ be the wave function defined by the fractional step method (see Section 
\ref{sect:fract}), and let $t\in[N\tau, (N+1)\tau)$. Then we have
\begin{multline}\label{eq:131}
\nabla\psi^\tau(t)=U(t)\nabla\psi_0
-i\frac{\tau}{\hbar}\sum_{k=1}^NU(t-k\tau)\pare{\phi^\tau_k\Lambda^\tau(k\tau-)}\\
-i\int_0^tU(t-s)F(s)\de s+\sum_{k=1}^NU(t-k\tau)r^\tau_k,
\end{multline}
where as before $U(t)$ is the free Schr\"odinger group, 
\begin{equation}
\|\phi^\tau_k\|_{L^\infty(\R^2)}\leq1, \quad\|r^\tau_k\|_{L^2(\R^2)}\leq\tau\|\psi_0\|_{H^1(\R^2)}
\end{equation}
and $F=\nabla(|\psi^\tau|^{p-1}\psi^\tau+V^\tau\psi^\tau)$.
\end{lemma}
\begin{remark}\label{rmk:25}
Now we can see more in details why the updating step in \eqref{eq:125} has been defined through the Lemma 
\ref{lemma:updat}. 
Indeed, one could possibly try to use the exact polar factorization as in the Lemma \ref{lemma:phase}, namely 
$\psi^\tau(k\tau+):=\phi^{(1-\tau)}\sqrt{\rho}$, where $\phi, \sqrt{\rho}$ are respectively the unitary factor and the amplitude of $\psi^\tau(k\tau-)$ (note that $\phi$ is uniquely determined $\sqrt{\rho}\de x-$a.e., and 
$|\phi|=1$, $\sqrt{\rho}\de x$-a.e.), as we anticipated in Section \ref{sect:fract}, see equation \eqref{eq:88}. 
In this case for the hydrodynamic quantities we get the exact updating formulas
\begin{align*}
|\psi^\tau(k\tau+)|^2&=|\psi^\tau(k\tau)|^2\\
\hbar\im(\conj{\psi^\tau}\nabla\psi^\tau)(k\tau+)&=(1-\tau)\hbar\im(\conj{\psi^\tau}\nabla\psi^\tau)(k\tau-).
\end{align*}
In order to obtain Strichartz-type estimates for the approximating sequence we need to compute $\nabla\psi^\tau$, the gradient of the approximated wave function. This calculation involves at time $t$, via the Duhamel's formula, the updates done at each $k\tau<t$:
\begin{align*}
\nabla\psi^\tau(t)=&U(t-N\tau)\sigma^\tau_{N\tau}U(\tau)\dotsc\sigma^\tau_\tau U(\tau)\nabla\psi_0\\
&-i\frac{\tau}{\hbar}U(t-N\tau)\sigma^\tau_{N\tau}\Lambda^\tau(N\tau-)
+\dotsc-i\frac{\tau}{\hbar}U(t-N\tau)\dotsc U(\tau)\phi^\tau_\tau\Lambda^\tau(\tau-)\\
&-i\int_{N\tau}^tU(t-s)F(s)\de s
-iU(t-N\tau)\sigma^\tau_{N\tau}\int_{(N-1)\tau}^{N\tau}U(N\tau-s)F(s)\de s\\
&+\dotsc
-iU(t-N\tau)\sigma^\tau_{N\tau}U(\tau)\dotsc\sigma^\tau_\tau\int_0^\tau U(\tau-s)F(s)\de s,
\end{align*}
where $\sigma^\tau_{k\tau}=(\phi^\tau_{k\tau})^{-\tau}$, $\phi^\tau_{k\tau}$ being the polar factor of 
$\psi^\tau(k\tau-)$.
Now, to use the group properties of the free Schr\"odinger operator, we need to pack all the $U(\cdot)$'s and all the $\sigma^\tau$'s. The estimates of these commutators $[U(\cdot), \sigma^\tau]$ are in general not true for $\sigma^\tau_{k\tau}$ which are only in $L^\infty(\R^2)$ (see \cite{Tay}).
\end{remark}
\begin{proof}
Since $\psi^\tau$ is solution of the Schr\"odinger-Poisson system in the space-time slab 
$[N\tau, (N+1)\tau)\times\R^2$, then we can write
\begin{equation}\label{eq:111}
\nabla\psi^\tau(t)=U(t-N\tau)\nabla\psi^\tau(N\tau+)-i\int_{N\tau}^tU(t-s)F(s)\de s,
\end{equation}
where $F$ is defined in the statement of the Lemma \ref{lemma:28}. Now there exists a piecewise smooth function 
$\theta_N$, as specified in the proof of the Lemma \ref{lemma:updat}, such that
\begin{equation*}
\psi(N\tau+)=e^{i(1-\tau)\theta_N}\sqrt{\rho_n},
\end{equation*}
and furthermore all the estimates of the Lemma \ref{lemma:updat} hold, with $\psi=\psi^\tau(N\tau-)$, 
$\tilde\psi=\psi^\tau(N\tau+)$ and $\eps=2^{-N}\tau\|\psi_0\|_{H^1(\R^2)}$. Therefore, we have
\begin{equation}\label{eq:112}
\nabla\psi^\tau(N\tau+)=\nabla\psi^\tau(N\tau-)
-i\frac{\tau}{\hbar}e^{i(1-\tau)\theta_N}\Lambda^\tau(N\tau-)+r^\tau_N,
\end{equation}
where $\|r^\tau_N\|_{L^2}\leq\tau\|\psi_0\|_{H^1}$.
By plugging \eqref{eq:112} into \eqref{eq:111} we deduce
\begin{multline*}
\nabla\psi^\tau(t)=U(t-N\tau)\nabla\psi^\tau(N\tau-)
-i\frac{\tau}{\hbar}U(t-N\tau)(e^{i(1-\tau)\theta_N}\Lambda^\tau(N\tau-))\\
+U(t-N\tau)r^\tau_N-i\int_{N\tau}^tU(t-s)F(s)\de s.
\end{multline*}
Let us iterate this formula, repeating the same procedure for $\nabla\psi^\tau(N\tau-)$, then \eqref{eq:131} holds.
\end{proof}
At this point we can use the formula \eqref{eq:131} to obtain Strichartz estimates for $\nabla\psi^\tau$, by applying the Theorem \ref{thm:strich} to each of the terms in \eqref{eq:131}. After some computations, the following result holds.
\begin{proposition}[Strichartz estimates for $\nabla\psi^\tau$]\label{prop:26}
For all $T>0$, let $\psi^\tau$ be the approximating wave function, then one has
\begin{equation}
\lplq{\nabla\psi^\tau}{q}{r}{[0,T]}\leq C(E_0^{\mez}, \|\rho_0\|_{L^1(\R^2)}, T)
\end{equation}
for each admissible pair of exponents $(q,r)$.
\end{proposition}
\begin{proof}
First of all, let us prove the Proposition \ref{prop:26}, for a small time $0<T_1\leq T$ and let 
$(q, r)$ be an admissible pair of exponents. We will choose $T_1>0$ later. Let $N$ be a positive integer such that $T_1\leq N\tau$. By applying the Theorem \ref{thm:strich} to the formula \eqref{eq:131}, we get
\begin{align*}
\LpLq{\nabla\psi^\tau}{q}{r}{[0,T_1]}\leq&\LpLq{U(t)\nabla\psi_0}{q}{r}{[0,T_1]}\\
&+\frac{\tau}{\hbar}\sum_{k=1}^N
\LpLq{U(t-k\tau)\pare{e^{i(1-\tau)\theta_k}\Lambda^\tau(k\tau-)}}{q}{r}{[0,T_1]}\\
&+\sum_{k=1}^N\LpLq{U(t-k\tau)r^\tau_k}{q}{r}{[0,T_1]}\\
&+\LpLq{\int_0^tU(t-s)F(s)\de s}{q}{r}{[0,T_1]}\\
=:A+B+C+D.
\end{align*}
Now we estimate term by term the above expression.
\newline
The estimate of $A$ is straightforward, since
\begin{equation*}
\|U(t)\nabla\psi_0\|_{L^q_tL^r_x([0,T_1]\times\R^2)}\lesssim\|\nabla\psi_0\|_{L^2(\R^2)}.
\end{equation*}
The estimate of $B$ follows from
\begin{multline}
\frac{\tau}{\hbar}\sum_{k=1}^N
\LpLq{U(t-k\tau)\pare{e^{i(1-\tau)\theta_k}\Lambda^\tau(k\tau-)}}{q}{r}{[0,T_1]}\\
\lesssim\tau\sum_{k=1}^N\|\Lambda^\tau(k\tau-)\|_{L^2(\R^2)}\lesssim T_1E_0^{\mez}.
\end{multline}
The term $C$ can be estimated in a similar way, namely
\begin{equation}
\sum_{k=1}^N\|U(t-k\tau)r^\tau_k\|_{L^q_tL^r_x}\lesssim\sum_{k=1}^N\|r^\tau_k\|_{L^2(\R^2)}
\lesssim T_1\|\psi_0\|_{H^1(\R^2)}.
\end{equation}
The last term is a little bit more tricky to estimate. First of all we decompose $F$ into three terms,
$F=F_1+F_2+F_3$, where $F_1=\nabla(|\psi^\tau|^{p-1}\psi^\tau)$, $F_2=\nabla V^\tau\psi^\tau$ and
$F_3=V^\tau\nabla\psi^\tau$.
By the Strichartz estimates (Theorem \ref{thm:strich}), we have
\begin{multline}
\LpLq{\int_0^tU(t-s)F(s)\de s}{q}{r}{[0,T_1]}\\
\lesssim\LpLq{F_1}{q_1'}{r_1'}{[0,T_1]}+\LpLq{F_2}{q_2'}{r_2'}{[0,T_1]}+\LpLq{F_3}{q_3'}{r_3'}{[0,T_1]},
\end{multline}
where $(q_i, r_i)$ are pairs of admissible exponents.
Now, we summarize the previous estimates by using \eqref{eq:131} in the following way
\begin{multline}
\|\nabla\psi^\tau\|_{\dot S^0([0,T_1]\times\R^2)}\lesssim\|\nabla\psi_0\|_{L^2(\R^2)}+T_1E_0^{\mez}
+T_1^\alpha\|\nabla\psi^\tau\|_{\dot S^0([0,T_1]\times\R^2)}^p\\
+T_1^{\mez}\|\psi_0\|_{L^2(\R^2)}^2\|\nabla\psi^\tau\|_{\dot S^0([0,T_1]\times\R^2)}\\
\lesssim(1+T)E_0^{\mez}
+T_1^\alpha\|\nabla\psi^\tau\|_{\dot S^0([0,T_1]\times\R^2)}^p
+T_1^{\mez}\|\psi_0\|_{L^2(\R^2)}^2\|\nabla\psi^\tau\|_{\dot S^0([0,T_1]\times\R^2)}.
\end{multline}
\begin{lemma}\label{lemma:38}
There exist $T_1(E_0, \|\psi_0\|_{L^2(\R^2)}, T)>0$ and $C_1(E_0, \|\psi_0\|_{L^2(\R^2)}, T)>0$, independent on $\tau$, such that
\begin{equation}\label{eq:212}
\|\nabla\psi^\tau\|_{\dot S^0([0, \tilde T]\times\R^2)}\leq C_1(E_0, \|\psi_0\|_{L^2(\R^2)}, T)
\end{equation}
for all $0<\tilde T\leq T_1(E_0, \|\psi_0\|_{L^2(\R^2)})$.
\end{lemma}
Let us recall that $E^\tau(t)=E_0$ and $\|\psi^\tau(t)\|_{L^2(\R^2)}=\|\psi_0\|_{L^2(\R^2)}$, hence we are in the situation in which we can repeat our argument on every time interval of length $T_1$, always depending on the same parameters $E_0, \|\psi_0\|_{L^2}$. Hence this yields to the following inequality on $[0, T]$
\begin{multline}\label{eq:213}
\|\nabla\psi^\tau\|_{\dot S^0([0, T]\times\R^2)}\\\leq C_1(E_0, \|\psi_0\|_{L^2}, T)
\pare{\quadre{\frac{T}{T_1}}+1}=C(\|\psi_0\|_{L^2}, E_0, T).
\end{multline}
\end{proof}
\emph{Proof of the Lemma \ref{lemma:38}}. Let us consider the non-trivial case $\|\psi_0\|_{L^2}>0$. Assume that $X\in(0, \infty)$ satisfies
\begin{equation}
X\leq A+\mu X+\lambda X^p=\phi(X),
\end{equation}
with $p>1$, $A>0$ and for all $0<\mu<1$, $\lambda>0$. Let $X_*$ be such that $\phi'(X_*)=1$, namely 
$X_*=\pare{\frac{1-\mu}{p\lambda}}^{\1{p-1}}$, hence one has $\phi(X_*)<X_*$ each time the following inequality is satisfied
\begin{equation}\label{eq:231}
\pare{\1{p^{\1{p-1}}}-\1{p^{\frac{p}{p-1}}}}\frac{(1-\mu)^{\frac{p}{p-1}}}{\lambda^{\1{p-1}}}>A.
\end{equation}
Therefore the convexity of $\phi$ implies that, if the condition \eqref{eq:231} holds, there exist two roots $X_{\pm}$, $X_+(\mu, \lambda, A)>X_*>X_-(\mu, \lambda, A)$, to the equation $\phi(X)=X$. It then follows either $0\leq X\leq X_-$, or $X\geq X_+$. In our case $\mu=T_1^{1/2}\|\psi_0\|_{L^2}^2$, 
$\lambda=T_1^\alpha$, $A=(1+T)E_0^{1/2}$, hence we assume
\begin{gather*}
\mu=T_1^{1/2}\|\psi_0\|_{L^2}^2<\mez\\
\lambda=T_1^\alpha=T_1^{\frac{5-p}{4}}
<\quadre{\frac{p^{-\1{p-1}}-p^{-\frac{p}{p-1}}}{2^{\frac{p}{p-1}}(1+T)E_0^{1/2}}}^{p-1}.
\end{gather*}
Therefore we choose
\begin{equation}
T_1:=\min\quadre{(2\|\psi_0\|_{L^2})^{-2}, 
\quadre{\frac{p^{-\1{p-1}}-p^{-\frac{p}{p-1}}}{2^{\frac{p}{p-1}}(1+T)E_0^{1/2}}}^{\frac{4(p-1)}{5-p}}}.
\end{equation}
Clearly we cannot have
\begin{equation*}
x_*=\quadre{\frac{1-T_1\|\psi_0\|_{L^2}}{pT_1^\alpha}}^{\1{p-1}}\leq x_+\leq
\|\nabla\psi^\tau\|_{\dot S^0([0, T_1]\times\R^2)},
\end{equation*}
since we get a contradiction as $T_1\to0$, hence
\begin{equation}
\|\nabla\psi^\tau\|_{\dot S^0([0, T_1]\times\R^2)}\leq X_-.
\end{equation}
\hfill $\square$\vspace{1cm}\newline
\begin{corollary}
For all $T>0$, let $\sqrt{\rho^\tau}, \Lambda^\tau$ be the approximating sequence defined by the fractional step method, then 
\begin{equation}
\lplq{\nabla\sqrt{\rho^\tau}}{q}{r}{[0, T]}+\lplq{\Lambda^\tau}{q}{r}{[0, T]}
\leq C(E_0^{\mez}, \|\rho_0\|_{L^1(\R^2)}, T),
\end{equation}
for each admissible pair of exponents $(q, r)$.
\end{corollary}
Unfortunately this is not enough to achieve the convergence of the quadratic terms. We need some additional compactness estimates on the sequence $\{\nabla\psi^\tau\}$ in order to apply Theorem 
\ref{thm:comp}. In particular we need some tightness and regularity properties on the sequence 
$\{\nabla\psi^\tau\}$, therefore we apply some results concerning local smoothing due to Vega 
\cite{V} and Constantin, Saut \cite{CS}.
 \begin{proposition}[Local smoothing for $\nabla\psi^\tau$]\label{prop:28}
For all $T>0$, let $\psi^\tau$ be defined as in the previous section, then it follows
\begin{equation}
\|\nabla\psi^\tau\|_{L^2([0,T];H^{1/2}_{loc}(\R^2))}\leq C(E_0, T, \|\rho_0\|_{L^1}).
\end{equation}
\end{proposition}
\begin{proof}
Using the Strichartz estimates obtained above, we can apply the Theorems \ref{thm:smooth1}, 
\ref{thm:smooth2} about local smoothing. Indeed, by using again the formula \eqref{eq:131} it follows
\begin{align*}
\|\nabla\psi^\tau\|_{L^2([0,T];H^{1/2}_{loc}(\R^2))}\lesssim
&\|\nabla\psi_0\|_{L^2(\R^2)}\\
&+\frac{\tau}{\hbar}\sum_{k=1}^N\|\Lambda^\tau(k\tau-)\|_{L^2(\R^2)}\\
&+\tau\sum_{k=1}^N\|\nabla\psi_{n_k}\|_{L^2(\R^2)}\\
&+\sum_{k=1}^N\norma{\nabla\psi_{n_k}-\nabla\psi^\tau(k\tau-)
+\frac{\tau}{\hbar}(\Lambda_{n_k}-\Lambda^\tau(k\tau-))}_{L^2(\R^2)}\\
&+\|F\|_{L^1([0,T];L^2(\R^2))}.
\end{align*}
\end{proof}Furthermore, since again we have the formula \eqref{eq:131}; for $t\in[N\tau, (N+1)\tau)$,
\begin{multline}
\nabla\psi^\tau(t)=U(t)\nabla\psi_0-i\int_0^tU(t-s)F(s)\de s\\
-i\frac{\tau}{\hbar}\sum_{k=1}^NU(t-k\tau)\quadre{\phi_k^\tau\Lambda^\tau(k\tau-)}
+\sum_{k=1}^NU(t-k\tau)r_k^\tau,
\end{multline}
holds (see Chapter \ref{sect:apr}), we can exploit the Strichartz estimates to obtain a bound like the one in \eqref{eq:2D_Strich}, 
\begin{equation}\label{eq:820}
\lplq{\nabla\psi^\tau}{q}{r}{[0, T]}\leq C(T, \|\psi_0\|_{L^2}, E_0),
\end{equation}
for every admissible pair $(q, r)$ of Schr\"odinger exponents in $\R^2$.
Hence, in order to obtain the sufficient compactness properties on the sequence of approximate solutions $\{\nabla\psi^\tau\}$, it only remains to recover the smoothing estimates for $\nabla\psi^\tau$. But this is straightforward since by theorems \ref{thm:smooth1}, \ref{thm:smooth2} we have
\begin{multline}
\|\nabla\psi^\tau\|_{L^2([0, T]; H^{1/2}_{loc}(\R^2))}\lesssim\|\nabla\psi_0\|_{L^2}
+\lplqs{|\psi^\tau|^{p-1}\nabla\psi^\tau}{1}{2}+\lplqs{V\nabla\psi}{1}{2}+\lplqs{\nabla V\psi}{1}{2}\\
\lesssim\|\nabla\psi_0\|_{L^2}+T^{\1{q_1'}}\lplqs{\psi}{\infty}{r_2(p-1)}^\alpha\lplqs{\nabla\psi}{q_1}{r_1}\\
+T^{\1{r'}}\lplqs{V}{\infty}{r}\lplqs{\nabla\psi}{r}{\frac{2r}{r-2}}
+T\lplqs{\nabla V}{\infty}{p}\lplqs{\psi}{\infty}{\frac{2p}{p-2}}\\
\lesssim\|\nabla\psi_0\|_{L^2}
+T^{\1{q_1'}}\lplqs{\psi}{\infty}{2}^{\frac{2}{r_2}}\lplqs{\nabla\psi}{\infty}{2}^{\frac{r_2(p-1)-2}{r_2}}\\
+T^{\1{r'}}\lplqs{V}{\infty}{r}\lplqs{\nabla\psi}{r}{\frac{2r}{r-2}}
+T\lplqs{\nabla V}{\infty}{p}\lplqs{\psi}{\infty}{2}^{1-\frac{2}{p}}\lplqs{\nabla\psi}{\infty}{2}^{\frac{2}{p}},
\end{multline}
where $(q_1, r_1)$ is a Sch\"odinger admissible pair in $\R^2$, $r_2$ is such that 
$2\leq r_2(p-1)<\infty$, $r$ is such that $V(0)\in L^r$ and $2<p<\infty$. Note that in the second inequality we used H\"older's inequality, while in the third one the Gagliardo-Nirenberg-Sobolev inequality.
Hence, by the estimate \eqref{eq:820} we can say
\begin{equation}
\|\nabla\psi^\tau\|_{L^2([0, T]; H^{1/2}_{loc}(\R^2))}\leq C(T, \|\psi_0\|_{L^2}, E_0).
\end{equation}
\newline
Hence, we can apply the theorem \ref{thm:comp} by Rakotoson, Temam; this means there exists 
$\psi\in L^\infty([0, T]; H^1(\R^2))\cap L^2([0, T]; H^{3/2}_{loc}(\R^2))$ such that
\begin{equation}
\psi^\tau\to\psi\qquad\textrm{in}\;L^\infty([0, T]; H^1(\R^2))
\end{equation}
strongly. Finally, by applying the consistency theorem \ref{thm:cons} we can state that $(\rho, J)$ defined by $(\rho, J):=(|\psi|^2, \hbar\im(\conj{\psi}\nabla\psi))$ is a finite energy weak solution to 
\eqref{eq:QHD}, \eqref{eq:QHD_IV}.
\appendix
\section{Polar Decomposition}
In this Section we will get into details of the main properties of polar factorization of a wave function, in particular we will prove Lemma \ref{lemma:phase}. 
\newline\newline
\emph{Proof of Lemma \ref{lemma:phase}.}
Let us consider a sequence $\{\psi_n\}\subset\cfun^\infty(\R^2)$, $\psi_n\to\psi$ in 
$H^1(\R^2)$, define the unitary factor
\begin{equation}\label{defn:reg-phas}
\phi_n(x)=\left\{\begin{array}{rl}
\frac{\psi_n(x)}{|\psi_n(x)|}&\textrm{if}\;\psi_n(x)\neq0\\
0&\textrm{if}\;\psi(x)=0.
\end{array}\right.
\end{equation}
Then, there exists $\phi\in L^\infty(\R^2)$ such that $\phi_n\wstar\phi$ in $L^\infty(\R^2)$ and 
$\nabla\psi_n\to\nabla\psi$ in $L^2(\R^2)$, hence
\begin{equation*}
\re\pare{\conj{\phi_n}\nabla\psi_n}\wlim\re\pare{\conj{\phi}\nabla\psi}\qquad\textrm{in}\;L^2(\R^2).
\end{equation*}
Since by \eqref{defn:reg-phas}, one has
\begin{equation*}
\re\pare{\conj{\phi_n}\nabla\psi_n}=\nabla|\psi_n|
\end{equation*}
a.e. in $\R^2$, it follows
\begin{equation*}
\nabla\sqrt{\rho_n}\wlim\re\pare{\conj{\phi}\nabla\psi}\qquad\textrm{in}\;L^2(\R^2).
\end{equation*}
Moreover, one has $\nabla\sqrt{\rho_n}\wlim\nabla\sqrt{\rho}$ in $L^2(\R^2)$, therefore
\begin{equation*}
\nabla\sqrt{\rho}=\re\pare{\conj{\phi}\nabla\psi},
\end{equation*}
where $\phi$ is a unitary factor of $\psi$.
\newline
The identity \eqref{eq:null-form} follows immediately from
\begin{multline}\label{eq:a_2}
\hbar^2\re(\d_j\conj{\psi}\d_k\psi)=\hbar^2\re((\phi\d_j\conj{\psi})(\conj{\phi}\d_k\psi))
=\hbar^2\re(\phi\d_j\conj{\phi})\re(\conj{\phi}\d_k\psi)\\
-\hbar^2\im(\phi\d_j\conj{\psi})\im(\conj{\phi}\d_k\psi)
=\hbar^2\d_j\sqrt{\rho}\d_k\sqrt{\rho}+\Lambda^{(j)}\Lambda^{(k)}.
\end{multline}
This identity is valid for any $\psi\in H^1(\R^3)$ since $\nabla\psi=0$ a.e. in the nodal region 
$\{\psi=0\}$.
\newline
Now we prove \eqref{eq:strong_conv}. Let us take a sequence $\psi_n\to\psi$ strongly in $H^1(\R^2)$, and consider $\nabla\sqrt{\rho_n}=\re\pare{\conj{\phi_n}\nabla\psi_n}$, 
$\Lambda_n:=\hbar\im\pare{\conj{\phi_n}\nabla\psi_n}$. 
As before, $\phi_n\wstar\phi$ in $L^\infty(\R^2)$, with $\phi$ a polar factor for $\psi$; then $\nabla\sqrt{\rho_n}\wlim\nabla\sqrt{\rho}$, 
$\re\pare{\conj{\phi_n}\nabla\psi_n}\wlim\re\pare{\conj{\phi}\nabla\psi}$, and 
$\nabla\sqrt{\rho}=\re\pare{\conj{\phi}\nabla\psi}$. Moreover, 
$\Lambda_n:=\hbar\im\pare{\conj{\phi_n}\nabla\psi_n}\wlim\hbar\im\pare{\conj{\phi}\nabla\psi}
=:\Lambda$.
To upgrade the weak convergence into the strong one, simply notice that by \eqref{eq:null-form} one has
\begin{multline*}
\hbar^2\|\nabla\psi\|_{L^2}^2=\hbar^2\|\nabla\sqrt{\rho}\|_{L^2}^2+\|\Lambda\|_{L^2}^2
\leq\liminf_{n\to\infty}\pare{\hbar^2\|\nabla\sqrt{\rho_n}\|_{L^2}^2+\|\Lambda_n\|_{L^2}^2}\\
=\liminf_{n\to\infty}\hbar^2\|\nabla\psi_n\|_{L^2}^2=\hbar^2\|\nabla\psi\|_{L^2}^2.
\end{multline*}
Finally to prove identity \eqref{eq:119} it suffices to notice that we can write
\begin{equation*}
\hbar\nabla\conj{\psi}\wedge\nabla\psi=\hbar(\phi\nabla\conj{\psi})\wedge(\conj{\phi}\nabla\psi),
\end{equation*}
where $\phi\in L^\infty(\R^2)$ is the polar factor of $\psi$ such that 
$\nabla\sqrt{\rho}=\re(\conj{\phi}\nabla\psi)$, $\Lambda:=\hbar\im(\conj{\phi}\nabla\psi)$, as in the Lemma 
\ref{lemma:phase}. By splitting 
$\conj{\phi}\nabla\psi$ into its real and imaginary part, we get the identity \eqref{eq:119}.
\hfill $\square$\vspace{1cm}\newline
Now we state a technical lemma which will be used in the next sections. It summarizes the results of this section we use later on and will be handy for the application to the fractional step method.
\begin{lemma}\label{lemma:updat}
Let $\psi\in H^1(\R^2)$, and let $\tau, \eps>0$ be two arbitrary (small) real numbers. Then there exists 
$\tilde\psi\in H^1(\R^2)$ such that
\begin{align*}
\tilde\rho&=\rho\\
\tilde\Lambda&=(1-\tau)\Lambda+r_\eps,
\end{align*}
where $\sqrt{\rho}:=|\psi|, \sqrt{\tilde\rho}:=|\tilde\psi|, \Lambda:=\hbar\im(\conj{\phi}\nabla\psi), 
\tilde\Lambda:=\hbar\im(\conj{\tilde\phi}\nabla\tilde\psi)$, $\phi, \tilde\phi$ are polar factors for $\psi, \tilde\psi$, respectively, and
\begin{equation*}
\|r_\eps\|_{L^2(\R^2)}\leq\eps.
\end{equation*}
Furthermore we have
\begin{equation}\label{eq:121}
\nabla\tilde\psi=\nabla\psi-i\frac{\tau}{\hbar}\phi^\star\Lambda+r_{\eps, \tau},
\end{equation}
where $\|\phi^\star\|_{L^\infty(\R^2)}\leq1$ and
\begin{equation}\label{eq:122}
\|r_{\eps, \tau}\|_{L^2(\R^2)}\leq C(\tau\|\nabla\psi\|_{L^2(\R^2)}+\eps).
\end{equation}
\end{lemma}
\begin{proof}
Consider a sequence $\{\psi_n\}\subset\cfun^\infty(\R^2)$ converging to $\psi$ in $H^1(\R^2)$, and define
\begin{equation*}
\phi_n(x):=\left\{\begin{array}{ll}
\frac{\psi_n(x)}{|\psi_n(x)|}&\textrm{if}\;\psi_n(x)\neq0\\
0&\textrm{if}\;\psi_n(x)=0
\end{array}\right.
\end{equation*}
as polar factors for the wave functions $\psi_n$. Since $\psi_n\in\cfun^\infty$, then $\phi_n$ is piecewise smooth, and $\Omega_n:=\{x\in\R^2\;:\;|\psi_n(x)|>0\}$  is an open set, with smooth boundary. Therefore we can say there exists a function $\theta_n:\Omega_n\to[0,2\pi)$, piecewise smooth in 
$\Omega_n$ and
\begin{equation*}
\phi_n(x)=e^{i\theta_n(x)},\qquad\textrm{for any}\;x\in\Omega_n.
\end{equation*}
Moreover, by the previous Lemma, we have $\phi_n\wstar\phi$ in $L^\infty(\R^2)$, where $\phi$ is a polar factor of $\psi$, and 
$\Lambda_n:=\hbar\im(\conj{\phi_n}\nabla\psi_n)\to\Lambda:=\hbar\im(\conj{\phi}\nabla\psi)$ in 
$L^2(\R^2)$. Thus there exists $n\in\N$ such that
\begin{equation*}
\|\psi-\psi_n\|_{H^1(\R^2)}+\|\Lambda-\Lambda_n\|_{L^2(\R^2)}\leq\eps.
\end{equation*}
Now we can define
\begin{equation}\label{eq:123}
\tilde\psi:=e^{i(1-\tau)\theta_n}\sqrt{\rho_n}.
\end{equation}
Let us remark that \eqref{eq:123} is a good definition for $\tilde\psi$, since outside $\Omega_n$, where $\theta_n$ is not defined, we have $\sqrt{\rho_n}=0$, and hence also $\psi_n=0$. Furthermore, 
\begin{align*}
\nabla\tilde\psi
=&e^{-i\tau\theta_n}\nabla\psi_n-i\frac{\tau}{\hbar}e^{i(1-\tau)\theta_n}\Lambda_n\\
=&\nabla\psi-i\frac{\tau}{\hbar}e^{i(1-\tau)\theta_n}\Lambda_n
+\tau\pare{\sum_{j=1}^\infty\frac{(-i\theta_n)^j\tau^{j-1}}{j!}}\nabla\psi_n\\
&+(\nabla\psi_n-\nabla\psi)-i\frac{\tau}{\hbar}e^{i(1-\tau)\theta_n}(\Lambda_n-\Lambda)
\end{align*}
and thus we obtain \eqref{eq:121} and \eqref{eq:122}, with $r_{\eps, \tau}$ given by
\begin{equation*}
r_{\eps, \tau}=\tau\pare{\sum_{j=1}^\infty\frac{(-i\theta_n)^j\tau^{j-1}}{j!}}\nabla\psi_n
+(\nabla\psi_n-\nabla\psi)-i\frac{\tau}{\hbar}e^{i(1-\tau)\theta_n}(\Lambda_n-\Lambda)
\end{equation*} 
Moreover
\begin{equation}
\tilde\Lambda=\hbar\im(e^{-i(1-\tau)\theta_n}\nabla\tilde\psi)
=(1-\tau)\Lambda_n
=(1-\tau)\Lambda+(1-\tau)(\Lambda_n-\Lambda)
\end{equation}
and clearly $r_\eps:=(1-\tau)(\Lambda_n-\Lambda)$ has $L^2$-norm less than $\eps$ by assumption.
\end{proof}
\subsection*{Acknowledgement}
We wish to thank prof. R. Danchin for having drawn our attention to an important technical remark needed to ensure the validity of the identity \eqref{eq:a_2}.

\end{document}